\documentclass[10pt,a4paper]{article}
\usepackage{xcolor}
\usepackage{graphicx}       % standard LaTeX graphics tool
\usepackage[latin1]{inputenc}
\usepackage[T1]{fontenc}
\usepackage{microtype} % good font tricks

\usepackage{cite}
\usepackage{array,colortbl}
\usepackage{amsmath,amsfonts,amssymb,bm,amsthm,mathtools,mathrsfs}
\DeclareFontFamily{U}{mathx}{\hyphenchar\font45}
\DeclareFontShape{U}{mathx}{m}{n}{
      <5> <6> <7> <8> <9> <10>
      <10.95> <12> <14.4> <17.28> <20.74> <24.88>
      mathx10
      }{}
\DeclareSymbolFont{mathx}{U}{mathx}{m}{n}
\DeclareFontSubstitution{U}{mathx}{m}{n}
\DeclareMathAccent{\widecheck}{0}{mathx}{"71}

\def\citep#1#2{\cite[{#1}]{#2}}

\usepackage{algorithm}
\usepackage{algorithmic}
\newcommand{\N}{{\mathbb{N}}} % natural numbers {1, 2, ...}
\newcommand{\R}{{\mathbb{R}}} % reals
\newcommand{\CC}{{\mathbb{C}}} % complex numbers
\newcommand{\NN}{{\mathbb{N}}} % natural numbers {1, 2, ...}
\newcommand{\RR}{{\mathbb{R}}} % reals
\newcommand{\EE}{{\mathbb{E}}}

\DeclareSymbolFont{bbold}{U}{bbold}{m}{n}
\DeclareSymbolFontAlphabet{\mathbbold}{bbold}

\newcommand{\bsb}{{\boldsymbol{b}}}

\newcommand{\bsy}{{\boldsymbol{y}}}

\newcommand{\bsalpha}{{\boldsymbol{\alpha}}}
\newcommand{\bsbeta}{{\boldsymbol{\beta}}}
\newcommand{\bsgamma}{{\boldsymbol{\gamma}}}

\newcommand{\bsmu}{{\boldsymbol{\mu}}}
\newcommand{\bsnu}{{\boldsymbol{\nu}}}

\newcommand{\bsrho}{{\boldsymbol{\rho}}}

\newcommand{\calE}{{\mathcal{E}}}
\newcommand{\calF}{{\mathcal{F}}}

\newcommand{\calI}{{\mathcal{I}}}

\newcommand{\calL}{{\mathcal{L}}}

\newcommand{\calN}{{\mathcal{N}}}
\newcommand{\calO}{{\mathcal{O}}}

\newcommand{\calR}{{\mathcal{R}}}
\newcommand{\calS}{{\mathcal{S}}}
\newcommand{\calT}{{\mathcal{T}}}

\newcommand{\calW}{{\mathcal{W}}}
\newcommand{\calX}{{\mathcal{X}}}
\newcommand{\calY}{{\mathcal{Y}}}

\newcommand{\fraku}{{\mathfrak{u}}}
\newcommand{\frakv}{{\mathfrak{v}}}

\newcommand{\frakT}{{\mathfrak{T}}}

\newcommand{\abs}[1]{\left\vert#1\right\vert}
\newcommand{\norm}[1]{\left\Vert#1\right\Vert}
\newcommand{\rd}{\,\mathrm{d}} % differential symbol with tiny space in front for use in integrals
\providecommand{\argmin}{\operatorname*{argmin}}

\author{Marcello Longo\footnote{Seminar for Applied Mathematics, ETH Z\"urich} , 
	Christoph Schwab\addtocounter{footnote}{0}\footnotemark[\value{footnote}] , 
	Andreas Stein\addtocounter{footnote}{0}\footnotemark[\value{footnote}]  \footnote{Corresponding author. Email: andreas.stein@sam.math.ethz.ch}
	\footnote{
		A.S. is partly funded by the ETH Foundation of Data Science iniative (ETH-FDS), and it is gratefully acknowledged.}
	}
\title{A-posteriori QMC-FEM error estimation \\ for Bayesian inversion and optimal control \\ with entropic risk measure}

\usepackage{graphicx}

\newcommand{\dbtilde}[1]{\tilde{\raisebox{0pt}[0.85\height]{$\tilde{#1}$}}}
\newcommand*{\eps}{\varepsilon}
\theoremstyle{plain}
  \newtheorem{theorem}{Theorem}
  \newtheorem{proposition}{Proposition}
  \newtheorem{lemma}{Lemma}
  \newtheorem{corollary}{Corollary}
\theoremstyle{definition}
  \newtheorem{definition}{Definition}
  
\theoremstyle{remark}
  \newtheorem{remark}{Remark}
\usepackage[only,llbracket,rrbracket]{stmaryrd} %double brackets

\begin{document}
    \maketitle
    \begin{abstract}
        We propose a novel a-posteriori error estimation technique
        where the target quantities of interest are ratios of high-dimensional integrals,
        as occur e.g. in PDE constrained Bayesian inversion and PDE constrained 
        optimal control subject to an entropic risk measure. 
        We consider in particular parametric, elliptic PDEs
        with affine-parametric diffusion coefficient, on high-dimensional parameter spaces.
        We combine our recent a-posteriori Quasi-Monte Carlo (QMC) error analysis,
        with Finite Element a-posteriori error estimation.
        The proposed approach yields a computable a-posteriori estimator 
        which is reliable, up to higher order terms.
        The estimator's reliability is uniform with respect to the PDE discretization,
        and robust with respect to the parametric dimension of the uncertain PDE input.
\begin{center}
Dedicated to Ian H. Sloan on his 85th anniversary
\end{center}
    \end{abstract}
%%%%%%%%%%%%%%%%%%%%%%%%%%%%%%%%%%%%%%%%%%%
    \section{Introduction}
\label{LoScSt:sec:Intro}
    The efficient numerical approximation of high-dimensional,
    parametric partial differential equations (PDEs for short)
    received increasing attention during the past years.
    In this work, we address two classes of high-dimensional numerical integration problems which
    arise in connection with data assimilation and PDE constrained optimization.
    We illustrate the abstract concepts for a parametric, linear elliptic PDE.
    The first class is the so-called \emph{Bayesian inverse problem} (BIP).
    There, we are interested in the posterior expectation of a (linear)
    functional of the solution $ u $ of a parametric PDE, conditional on observation data
    subject to additive, centered Gaussian observation noise \cite{LoScSt:Stu10,LoScSt:SS14}. 
    See also \cite{LoScSt:HKS21}.
    A second problem class is PDE-constrained optimization.
    Specifically,
    the optimal control problem (OCP) of parametric PDEs under an \emph{entropic risk} measure \cite{LoScSt:FK11},
    where the state variable satisfies a parametric PDE constraint \cite{LoScSt:GKKSS22}. 
     %\todo{[More refs., more text on opt. contr.]}.

    When numerically approximating solutions of a BIP or an OCP,
    it is essential to quantify the error due to the numerical discretization
    in order to meet a prescribed numerical tolerance without wasting computational resources.
    As solving PDEs exactly is in general not possible,
    discretizations such as Finite Element Methods (FEM) must be used instead.
    Additionally, the parametric uncertainty in the forward PDE model
    is passed on to the solution, 
    and this must be taken into account both in the computation and in the error estimation.
    This justifies the need for an \emph{a-posteriori} error analysis.
We introduce a novel class of QMC-FEM algorithms for BIPs and OPCs with a-posteriori error control.
We derive numerical a-posterior error estimates for ratio estimators with respect to both, 
the spatial and parametric domain errors. 
This allows efficient computation of quantities of interest for any prescribed tolerance, 
and with error bounds in the parametric domain that are uniform with respect to the parameter dimension $s$.
In particular, we assume that the uncertain PDE coefficients are affine-parametric with parameter vector
    $ \bsy \in U := \left[-\frac{1}{2},\frac{1}{2} \right]^s $ 
    where 
    $ s\in \NN $, $ s\gg 1 $, e.g.\ \eqref{LoScSt:eq:Param} below.
    We can employ suitable quasi-Monte Carlo (QMC) rules to approximate integrals over $ U $.
    Here, we select extrapolated polynomial lattice (EPL) rules as 
    first introduced in \cite{LoScSt:DGY19,LoScSt:DLS22}.
    This choice is motivated by the deterministic nature of their quadrature nodes $ P_m $, $ |P_m| = b^m $
    for some prime $ b $ \cite{LoScSt:N92},
    and good convergence properties under quantified parametric regularity
    of the integrand functions with respect to $ \bsy \in U $,
    uniformly in the dimension $ s $ \cite{LoScSt:DGY19}.
    Moreover, it was shown in \cite{LoScSt:DLS22,LoScSt:ML22DD}
    that under assumptions, EPL quadratures allow for computable a-posteriori quadrature error estimators 
    that are asymptotically exact as $ m\to \infty $,
    with dimension robust ratio between estimated and actual quadrature error.
	
    Both, BIPs and OCPs for PDEs with parametric input take the form
    \begin{equation}\label{LoScSt:ratios}
        \frac{Z'}{Z} \in \calY, \quad \text{ where } Z = \int_{U} \Theta(\bsy) \rd \bsy, \quad Z' = \int_{U} \Theta'(\bsy) \rd \bsy,
    \end{equation}
    for some suitable integrable functions
    $ \Theta\colon U \to \RR $ and $ \Theta' \colon U\to \calY $,
    where $ \calY $ is a separable Hilbert space and $ Z, Z' $ are Bochner integrals 
    with respect to a product measure $\rd\bsy$ on the possibly 
    high-dimensional parameter space $U$.
    In particular,
    we have $ \calY = \RR $ for the BIP case and $ \calY \in L^2(D) $
    for the OCP case, with $ D $ being the physical domain of the considered PDE.
    Approximating the high-dimensional integrals with averages over polynomial lattices $ P_m \subset U$
    yields a first approximation
    \begin{equation}\label{LoScSt:ratiosqmc}
        \frac{Z_m'}{Z_m} \in \calY, \quad \text{ where } Z_m 
         = \frac{1}{b^m} \sum_{\bsy \in P_m} \Theta(\bsy), 
         \quad 
         Z_m' = \frac{1}{b^m} \sum_{\bsy \in P_m} \Theta'(\bsy).
    \end{equation}
    Then, since the integrands $ \Theta, \Theta' $ depend on the solution
    of a $\bsy$-parametric PDE, % that does not admit exact solutions,
    we discretize the parametric PDEs for $\bsy\in P_m$,
    resulting in computable, parametric integrand functions 
    $ \Theta_h(\bsy), \Theta_h'(\bsy) $
    and in the computable estimates
    \begin{equation}\label{LoScSt:ratiosqmcfem}
    \frac{Z_{m,h}'}{Z_{m,h}} \in \calY, \quad \text{ where }
    Z_{m,h} = \frac{1}{b^m} \sum_{\bsy \in P_m} \Theta_h(\bsy), \quad
    Z_{m,h}' = \frac{1}{b^m} \sum_{\bsy \in P_m} \Theta'_h(\bsy).
    \end{equation}
    Here, the parameter $ h > 0 $ denotes the meshwidth of conforming
    Lagrangian Finite Element discretizations.
    We present a
    \emph{computable a-posteriori estimator for the combined Finite Element discretization and quadrature error}
    \begin{equation}\label{LoScSt:eq:err}
        \text{err}  = \norm{\frac{Z'}{Z} - \frac{Z'_{m,h}}{Z_{m,h}}}_{\calY}.
    \end{equation}
    In the rest of this section we introduce the setting and
    we describe the two problems of interest, namely the BIP and the OCP with entropic risk measure.
    Section \ref{LoScSt:sec:QMCapost} and Section \ref{LoScSt:sec:GalDis} are devoted to the QMC and the FEM a-posteriori error analysis,
    respectively, and these results will be combined 
    in Section~\ref{LoScSt:sec:combined}. 
    We present numerical experiments %for a BIP 
    in Section~\ref{LoScSt:sec:numerics}
    and summary and conclusions 
    in Section~\ref{LoScSt:sec:conclusion}.
    %%%%%%%%%%%%%%%%%%%%%%%%%%%%%%%%
    \subsection{Affine-Parametric Forward PDE}
    \label{LoScSt:sec:FwdPbm}
    %%%%%%%%%%%%%%%%%%%%%%%%%%%%%%%%
    For brevity of presentation, we consider a model, linear elliptic PDE with
    homogeneous Dirichlet boundary conditions.
    The corresponding QMC-FEM error analysis for this (forward) model with EPL rules has been obtained 
    in \cite{LoScSt:DLS22, LoScSt:ML22DD}, and, more generally, for parametric holomorphic operators in \cite{LoScSt:DLS16, LoScSt:DLS22}. 
    The results therein are the foundation for our a-posteriori error analysis
    of the QMC ratio estimators in Subsections~\ref{LoScSt:sec:BIP} and~\ref{LoScSt:sec:OCEntrRisk}.
    
    %in divergence form.
    Given a bounded polygon $ D \subseteq \RR^2 $
    and a parameter sequence $ \bsy \in U $, $ s\in \NN $,
    consider the following parametric, linear, second order elliptic PDE in variational form:
    find $ u(\cdot,\bsy) \in \calX = H_0^1(D) $ such that
    \begin{equation}
        \label{LoScSt:eq:PDE}
        \int_D a(x,\bsy) \nabla_x u(x,\bsy) \cdot \nabla_x v(x) \rd x = \int_D f(x) v(x) \rd x \qquad \forall v\in \calX.
    \end{equation}
    We assume that $ a $ is affine-parametric, namely that we are given a family of functions
    $ \{\psi_j\}_{j\in \NN_0} \subseteq L^{\infty}(D)$
    such that, with $ \operatorname{essinf} \psi_0 > \kappa > 0 $ and $ b_j := \frac{1}{\kappa}\norm{\psi_j}_{L^{\infty}(D)} $,
    we have
    \begin{equation}
        \label{LoScSt:eq:Param}
        a(x,\bsy) = \psi_0(x) + \sum_{j = 1}^{s} y_j\psi_j, \qquad \sum_{j\ge 1}b_j < 2.
    \end{equation}
    Then $ \operatorname{essinf} a(\cdot,\bsy) > \operatorname{essinf} \psi_0  - \kappa > 0 $
    for all $ \bsy\in U $.
    By the Lax-Milgram lemma,  the parametric weak solution $ u(\cdot,\bsy) \in \calX $
    is well defined for any $ f \in \calX^* = H^{-1}(D) $, where $ {}^* $ denotes the topological dual.
    To justify the a-posteriori QMC error analysis of Section \ref{LoScSt:sec:QMCapost},
    we will additionally require the summability
    \begin{equation}\label{LoScSt:eq:ellp}
         \bsb = (b_j)_{j\ge 1} \in \ell^p(\NN) \quad  \text{ for some } p \in (0,1/2).
    \end{equation}
    The condition $p \in (0,1/2)$ is imposed to derive the second part of Theorem~\ref{LoScSt:thm:apostqmc} below (using \cite[Theorem 4.1]{LoScSt:DLS22}), i.e. to bound the difference of two successive QMC estimators.
    
    For the FEM approximation, we consider conforming subspaces%
    \footnote{In practice, $ h $ either parametrizes the local mesh-size $ \max_{T\in \calT_h} |T|^{1/2} $,
    for quasi-uniform collections of partitions,
    or it relates to the refinement level in case of adaptive refinement \cite{LoScSt:NSV09}.}
    $ \calX_{h}\subseteq \calX $ $ h \in H \subseteq (0,\infty) $, $ \dim(\calX_{h}) < \infty $,
    that are linked to shape-regular, simplicial partitions $ \{\calT_h\}_{h\in H} $ of $ D $
    \cite[Section 8]{LoScSt:EG21BB}.
    Assume that the resulting spaces are nested and conforming,
    that is $ \calX_{h} \subseteq \calX_{h'} $ for any $ h,h'\in H, h>h' $ and that $ H $ accumulates at $ 0 $.
    We construct the Galerkin discretizations $ u_h(\bsy) \in \calX_h $ of \eqref{LoScSt:eq:PDE},
    by solving
    \begin{equation}
        \label{LoScSt:eq:PDE2}
        \int_D a(x,\bsy) \nabla_x u_h(x,\bsy) \cdot \nabla_x v(x) \rd x = \int_D f(x) v(x) \rd x \qquad \forall v\in \calX_h.
    \end{equation}
    To simplify notation, we write
    $ a(\bsy) = a(\cdot,\bsy) $ and $ u(\bsy) = u(\cdot,\bsy) $
    and we omit the variable $ x $ for $ \nabla_x = \nabla $ 
    and $ \operatorname*{div}_x = \operatorname*{div} $.
    We remark that the assumptions on $a$ given in~\eqref{LoScSt:eq:Param} imply 
    $a(\bsy)\in L^{\infty}(D)$ for all $\bsy\in U$, but not necessarily $a(\bsy)\in W^{1,\infty}(D)$. 
    The latter is in turn necessary to control the (a-posteriori) FEM approximation error, hence the condition
    $a(\bsy)\in W^{1,\infty}(D)$ will be imposed for the BIP and OCP in Section \ref{LoScSt:sec:GalDis} below.
    %%%%%%%%%%%%%%%%%%%%%%%%%%%%%%%%%%%%%%%%%%%%%%%%%%%%%%%%%%%%%%%%
    \subsection{Bayesian inverse problem (BIP)}
    \label{LoScSt:sec:BIP}
    %%%%%%%%%%%%%%%%%%%%%%%%%%%%%%%%%%%%%%%%%%%%%%%%%%%%%%%%%%%%%%%%
    Let $ X = \{a \in L^{\infty}(D) : \ \operatorname{essinf} a > 0 \}$
    and fix $ f \in \calX^* $.
    Then, we can
    define the data-to-solution map $ \calS \colon  X \to \calX $ for the forward problem \eqref{LoScSt:eq:PDE}.
    We also define
    the observation functional $ \calO \in (\calX^*)^{K} $, with a finite number $ K \in\N $
    of observations (e.g.\ representing sensors),
    and a goal functional (also called quantity of interest) $ G\in \calX^* $.
    We define the prior measure $ \pi_0 $ to be the uniform distribution on $U$.

    The observations $ \calO(\calS(a)) $ are assumed to be
    additionally subject to additive observation noise $ \eta $,
    which we assume 
    to be centered Gaussian, i.e., $ \eta \sim \calN(0,\Gamma) $
    for some known, nondegenerate covariance matrix $ \Gamma \in \RR^{K\times K}$.
    In other words,
    we assume given noisy observation data $ \delta \in \RR^K$
    modeled as
    \begin{equation}
        \delta = \calO(\calS(a)) + \eta \in L_{\Gamma}^2(\RR^K).
    \end{equation}
    We consider the Bayesian inverse problem of recovering the expected value of $ G(u) $,
    given observation data $ \delta $, that is $ \EE^{\pi_0}[G(u)| \delta] $.
    By Bayes' theorem \cite{LoScSt:Stu10},
    the posterior distribution $ \pi^{\delta} $ of $ \bsy |\delta $
    is absolutely continuous with respect to $ \pi_0 $ and
    its Radon-Nikodym derivative with respect to the prior $\pi_0$
    is
    \begin{equation}
        \frac{\rd \pi^{\delta}}{\rd \pi_0}(\bsy) = \frac{\Theta(\bsy)}{Z},
    \end{equation}
    where
    $ \Theta(\bsy) := \exp(-\frac12 |\delta - \calO(\calS(a(\bsy)))|_{\Gamma}^2 )
                    = \exp(-\frac12 |\delta - \calO(u(\bsy))|_{\Gamma}^2 ) $
    denotes the likelihood with the observation noise
    covariance-weighted, data-to-observation misfit,
    where
    $ \abs{x}_{\Gamma}^2 := x^{\top} \Gamma^{-1}x $
    and
    $ Z $
    is defined in \eqref{LoScSt:ratios}.
    As $ \Theta(\bsy) > 0 $ for all $ \bsy\in U $, $ Z>0 $.
    In the present setting,
    Bayesian inversion amounts to the numerical evaluation of the posterior mean
    \begin{equation*}
        \EE^{\pi^{\delta}}[G(u)] = \frac{1}{Z} \int_{U} G(u(\bsy)) \Theta(\bsy)\rd \bsy.
    \end{equation*}
    This is \eqref{LoScSt:ratios} upon setting
    $ \Theta'(\bsy) := G(u(\bsy)) \Theta(\bsy) $ and $ \calY = \RR $.
    Define the FE solution operator $ \calS_h\colon X \to \calX_h $ as the
    mapping $ \calS_h a(\bsy) = u_h(\bsy) $ via \eqref{LoScSt:eq:PDE2}.
    The FE approximations of $ \Theta,\Theta' $ used in \eqref{LoScSt:ratiosqmcfem} 
    are then
    $ \Theta_{h} = \exp(-\frac12 |\delta - \calO(u_h(\bsy))|_{\Gamma}^2 ) $
    and
    $ \Theta'_{h} = G(u_h(\bsy)) \Theta_{h}(\bsy)$, respectively.
	
	QMC-ratio estimators to approximate $\EE^{\pi^{\delta}}[G(u)]$ in BIPs have been investigated in a multi-level setting, among others, in \cite{LoScSt:DGLS17}. We utilize some of the results in this reference, to complement the (single-level) estimator with a-posteriori QMC-FEM  error control.
	
    \subsection{Optimal control with entropic risk measure (OCP)}
    \label{LoScSt:sec:OCEntrRisk}
    Let $ \calY = L^2(D) $, assume a parameter independent target state $ \hat{u}\in \calY $
    and a nonempty, closed and convex 
    set $ X \subseteq \calY $ of admissible controls.
    Throughout the rest of the paper,
    we identify $ \calY $ with its dual via Riesz representation
    and write $ \langle \cdot,\cdot\rangle $ for the inner product on $ \calY $.
    Once the affine parametric diffusion coefficient $ a(\bsy) $ is fixed,
    $ \eqref{LoScSt:eq:PDE} $ defines a linear solution operator 
    $ \calL^{\bsy}\colon \calY \to \calY $ by 
    $ \calL^{\bsy} f = \iota \circ u(\bsy) $ for all $ f\in \calY $, 
    where $ \iota $ denotes the continuous embedding $ \calX \subset \calY $.
    In particular, we view $ u(\bsy) $ as a function of the right-hand side $ f $ of \eqref{LoScSt:eq:PDE}.
    For a function $ \Phi\colon U\to \RR $ and some $\theta\in(0,\infty)$, 
    the entropic risk measure \cite{LoScSt:KS18} is defined by
    \begin{equation}
        \calR(\Phi) = \frac{1}{\theta} \log\left (\int_{U}\exp(\theta \Phi(\bsy)) \rd \bsy \right ).
    \end{equation}
    The entropic risk is especially relevant when favoring a risk averse behavior \cite{LoScSt:FK11}.
    We consider the following minimization problem, for fixed constants $ \alpha_1,\alpha_2 > 0 $
    \begin{equation}\label{LoScSt:ocexact}
        f^*:= \argmin_{f\in X} J(f), \qquad J(f) := \calR(\tfrac{\alpha_1}{2}\norm{\calL^{\bsy} f - \hat{u}}_{\calY}^2) + \tfrac{\alpha_2}{2} \norm{f}_{\calY}^2.
    \end{equation}
    Due to
    convexity of $ \calR $ and $ \alpha_2>0 $,
    the functional $ J $ is strongly convex so that \eqref{LoScSt:ocexact}
    is a well-posed minimization problem \cite{LoScSt:KS18,LoScSt:GKKSS22}.
	The weights $\alpha_1,\alpha_2$ are imposed to balance the entropic risk of the residual $\norm{\calL^{\bsy} f - \hat{u}}_{\calY}$ against the magnitude of the control $f$ via the regularization term $\tfrac{\alpha_2}{2} \norm{f}_{\calY}^2$.
	
    Define the shorthand notation $ \Phi_f(\bsy) = \tfrac{\alpha_1}{2}\norm{\calL^{\bsy} f - \hat{u}}_{\calY}^2 $ and the adjoint state given by $  q_f(\bsy) = \alpha_1 \calL^{\bsy} (\calL^{\bsy}f - \hat{u}) \in \calY$.
    Under the above conditions on $ X $, \eqref{LoScSt:ocexact} is equivalent to the inequality $ \langle J'(f^*),  f - f^*\rangle \ge 0 $ for all $ f\in X $, where in analogy with \cite[Lemma 3.6]{LoScSt:GKKSS22}  the Fr\'echet derivative $ J'(f) \in \calY $ of $ J $ at $ f\in X $ is
    \begin{equation}\label{LoScSt:derivJ}
        J'(f) = \frac{1}{\int_{U}\exp(\theta \Phi_f(\bsy) ) \rd \bsy} \int_U \exp(\theta \Phi_f(\bsy) ) q_f(\bsy) \rd \bsy + \alpha_2 f.
    \end{equation}
    Next, we replace in \eqref{LoScSt:ocexact} the integral over $ U $ by QMC rules and the exact solution operator $ \calL^{\bsy} $ by the Galerkin solution $ \calL_{h}^{\bsy} \colon \calY \to \calX_{h} $ defined by $ \calL_{h}^{\bsy}f = u_h(\bsy) $. Then, we obtain the discrete formulation
    \begin{equation}\label{ocapx}
        f_{m,h}^*:= \argmin_{f\in X} J_{m,h}(f), \qquad J_{m,h}(f) := \calR_{m}(\tfrac{\alpha_1}{2}\norm{\calL_h^{\bsy} f - \hat{u}}_{\calY}^2) + \tfrac{\alpha_2}{2} \norm{f}_{\calY}^2,
    \end{equation}
    where $ \calR_{m}(\Phi) = \frac{1}{\theta}\log\left(\frac{1}{b^m}\sum_{\bsy \in P_m} \exp(\theta\Phi(\bsy)) \right ) $ is again convex, due to positivity of the QMC quadrature weights.
    The derivative $ J'_{m,h}(f) \in \calY $ of $ J_{m,h} $ is
    analogous to \eqref{LoScSt:derivJ}, again replacing the integrals by sample averages over $ P_m $, $ \Phi_{h,f}(\bsy) = \tfrac{\alpha_1}{2}\norm{\calL_{h}^{\bsy} f - \hat{u}}_{\calY}^2 $ and $ q_{h,f}(\bsy) = \alpha_1 \calL_{h}^{\bsy} (\calL_{h}^{\bsy}f - \hat{u}) \in \calY$.
    The next proposition recasts the error in the approximation $ f_{m,h}^* \approx f^* $ to the form \eqref{LoScSt:eq:err}. Whenever it does not cause confusion, we will write $ q(\bsy) = q_{f_{m,h}^*}(\bsy)  $ and $ q_h(\bsy) = q_{h,f_{m,h}^*}(\bsy) $.
    \begin{proposition}\label{LoScSt:prop:upboundfstar}
        For $ \calY = L^2(D) $, assume $ Z,Z' $ in \eqref{LoScSt:ratios} are defined by 
        $ \Theta(\bsy) = \exp(\theta \Phi_{f_{m,h}^*}(\bsy)) $ and $ \Theta'(\bsy) = q(\bsy) \Theta(\bsy) $. 
        Similarly, let $ Z_{m,h}, Z'_{m,h} $ in \eqref{LoScSt:ratiosqmcfem} 
        be defined by $ \Theta_{h}(\bsy) = \exp(\theta \Phi_{h,f_{m,h}^*}(\bsy)) $ 
        and $ \Theta_{h}'(\bsy) = q_{h}(\bsy) \Theta_{h}(\bsy) $. 

	Then,
        \begin{equation}\label{LoScSt:upboundfstar}
            \norm{f^* - f_{m,h}^*}_{\calY} \le \frac{1}{\alpha_2} \norm{  J'(f_{m,h}^*) -J_{m,h}'(f_{m,h}^*) }_{\calY} = \frac{1}{\alpha_2}\norm{\frac{Z'}{Z} - \frac{Z'_{m,h}}{Z_{m,h}}}_{\calY}.
        \end{equation}
    \end{proposition}
\begin{proof}
    The inequalities $ \langle J'(f^*),  f - f^*\rangle \ge 0 $ and $ \langle J_{m,h}'(f_{m,h}^*),  f - f_{m,h}^*\rangle \ge 0 $ for all $ f\in X $ imply that  $ \langle J_{m,h}'(f_{m,h}^*) - J'(f^*),  f^* - f_{m,h}^*\rangle \ge 0 $.
    Moreover, strong convexity of $ J $ yields the relation $ \langle J'(f^*) - J'(f_{m,h}^*) - \alpha_2 (f^* - f_{m,h}^*), f^* - f_{m,h}^* \rangle \ge 0 $.
    Thus, we get
    \begin{align*}
        \alpha_2 \norm{f^* - f_{m,h}^*}_{\calY}^2 &\le \langle  J_{m,h}'(f_{m,h}^*) -  J'(f^*) + \alpha_2 (f^* - f_{m,h}^*), f^* - f_{m,h}^* \rangle \\
        & \le \langle  J_{m,h}'(f_{m,h}^*) -  J'(f_{m,h}^*), f^* - f_{m,h}^* \rangle \\
        &\le \norm{J_{m,h}'(f_{m,h}^*) -  J'(f_{m,h}^*)}_{\calY} \norm{f^* - f_{m,h}^*}_{\calY},
    \end{align*}
    which implies the inequality in~\eqref{LoScSt:upboundfstar}. 
    The equality follows by substituting 
    $ J_{m,h}'(f_{m,h}^*) = \frac{Z_{m,h}'}{Z_{m,h}} + \alpha_2 f_{m,h}^* $ and \eqref{LoScSt:derivJ}.
\end{proof}
\begin{remark}\label{LoScSt:rmk:fstarmh}
    Note that $ \Theta, \Theta' $ as defined in Proposition \ref{LoScSt:prop:upboundfstar}, 
    and hence $ Z, Z' $, also implicitly depend on $ h $ via the discrete minimizer $ f_{m,h}^* $.
    In particular, the exact minimizer $ f^* $ does not appear in the right hand side of \eqref{LoScSt:upboundfstar}.
    This fact will be crucial for the ensuing a-posteriori error estimation methodology.
\end{remark}
%%%%%%%%%%%%%%%%%%%%%%%%%%%%%%%%%%%%%%%%%%%%%%%%%%%%%%%%%%%%%%%%%%%%%%%%%%%%%%%%%%%%%%%%%%%%%%%%%5
    \section{A-posteriori QMC error estimation}
    \label{LoScSt:sec:QMCapost}
%%%%%%%%%%%%%%%%%%%%%%%%%%%%%%%%%%%%%%%%%%%%%%%%%%%%%%%%%%%%%%%%%%%%%%%%%%%%%%%%%%%%%%%%%%%%%%%%%5
We develop computable a-posteriori error estimators for the PDE discretization error and for
the QMC-quadrature error, the latter being reliable independent of the integration dimension $s$.
%%%%%%%%%%%%%%%%%%%%%%%%%%%%%%%%%%%%%%%%%%%%%%%%%%%%%%%%%%%%%%%%%%%%%%%%%%%%%%%%%%%%%%%%%%%%%%%%%5
    \subsection{Parametric regularity}
    \label{LoScSt:sec:ParReg}
%%%%%%%%%%%%%%%%%%%%%%%%%%%%%%%%%%%%%%%%%%%%%%%%%%%%%%%%%%%%%%%%%%%%%%%%%%%%%%%%%%%%%%%%%%%%%%%%%5
    To leverage the results from \cite{LoScSt:DLS22,LoScSt:ML22DD} 
    and overcome the curse of dimensionality,
    we need to quantify the regularity with respect to $ \bsy \in U$.
    To this end,
    we write $ |\bsnu| = \sum_j \nu_j $, $ \operatorname*{supp}(\bsnu) = \{j : \nu_j\neq 0\} $
    and, given smooth 
    $ F \colon U \to \calY $ and $ \bsnu \in \calF := \{\bsnu \in \NN_0^{\NN} : |\operatorname*{supp}(\bsnu)| < \infty \} $
    we introduce the multi-index notation 
    $ \partial_{\bsy}^{\bsnu} F(\bsy) = \prod_{j\in \operatorname*{supp}(\bsnu)}\partial_{y_j}^{\nu_j}F(\bsy) $
    for the derivatives with respect to $ \bsy $.
    Moreover, given $ \bsbeta = (\beta_j)_{j\in \NN} \in\ell^p(\NN) $ for some $ p \in (0,1) $, $ n\in \NN $ and $ c > 0 $, 
    we define the SPOD weights $ \bsgamma = (\gamma_{\fraku})_{\fraku \subseteq \{1:s\}} $ via
    \begin{equation}\label{LoScSt:def:spod}
        \gamma_{\fraku} = \sum_{\bsnu \in \{1:\alpha\}^{|\fraku|}} (|\bsnu| + n)! \prod_{j\in \fraku} c\beta_j^{\nu_j}.
    \end{equation}
    \begin{definition}
        Let $ \calY $ be a separable Hilbert space, $ \alpha\in \NN $, $ \alpha \ge 2 $. 
       We define the weighted unanchored Sobolev space $ \calW_{s,\alpha,\bsgamma} $ 
       with dominating mixed smoothness  as the completion of $ C^{\infty}(U,\calY) $ with respect to the norm
        \begin{align*}
            \norm{F}_{s,\alpha,\bsgamma} &:= \max_{\fraku\subseteq \{1:s\}} \gamma_{\fraku}^{-1} \sum_{\frakv \subseteq \fraku} \sum_{\bsnu_{\fraku\setminus\frakv} \in \{1:\alpha\}^{|\fraku\setminus\frakv|}} \\ &\qquad\qquad\qquad
            \int_{[-\frac12,\frac12]^{|\frakv|}} \norm{\int_{[-\frac12,\frac12]^{s-|\frakv|}} \partial_{\bsy}^{(\bsnu_{\fraku\setminus\frakv},\bsalpha_{\frakv})}F(\bsy) \rd \bsy_{\{1:s\}\setminus\frakv} }_{\calY}\rd \bsy_{\frakv},
        \end{align*}
    where $ \bsmu = (\bsnu_{\fraku\setminus\frakv},\bsalpha_{\frakv}) \in \calF $ is such that $ \mu_j = \nu_j $ if $ j\in \fraku\setminus\frakv $, $\mu_j = \alpha $ if $ j\in \frakv $ and $ \mu_j = 0 $ otherwise.
    The inner integral is interpreted as a Bochner integral.
    \end{definition}
    The relevance of the space $ \calW_{s,\alpha,\bsgamma} $ in our context is justified by the following result, 
    which will be the starting point of our analysis.
    This is a consequence of the so-called component-by-component (CBC) construction as described in \cite{LoScSt:DLS22,LoScSt:ML22DD},
    which takes as input $ s, \alpha, \bsgamma $ and $ m $ and returns a polynomial lattice $ P_m $.
    %The CBC construction is based on \todo{other works\cite{bibid}}.
    \begin{theorem}\label{LoScSt:thm:apostqmc}
        Let $ \alpha\in \NN, \alpha \ge 2 $ and $ F\colon U \to \calY $ for some separable Hilbert space $ \calY $ 
        be such that $ F\in \calW_{s,\alpha,\bsgamma} $ for some weights $ \bsgamma $ of the form \eqref{LoScSt:def:spod} 
        for $ \bsbeta \in \ell^p(\NN) $, $ p \in (0,1/2]$. 
        Then, there exists a sequence $ (P_m)_{m\in \NN} $ of polynomial lattice rules such that as $ m\to \infty $
        \begin{equation*}
            {\mathfrak E_{m}(F)}:=\int_U F - \frac{1}{b^m} \sum_{\bsy\in P_m}F(\bsy) = \norm{F}_{s,\alpha,\bsgamma} O(b^{-m}),
        \end{equation*}
        as well as for all $ \eps>0 $
        \begin{equation*}
            {\mathfrak E_{m}(F)}  
            = \frac{1}{b^m}\sum_{\bsy\in P_m}F(\bsy) - \frac{1}{b^{m-1}} \sum_{\bsy\in P_{m-1}} F(\bsy) + \norm{F}_{s,\alpha,\bsgamma} O(b^{-2m + \eps}).
        \end{equation*}
    Here the constants hidden in $ O(\cdot) $ are independent of $ s,m\in \NN $ and $ F $, 
    but depend on $ \eps $ and $ \bsgamma $. 
    In addition, 
    the point sets $ (P_m)_{m\in \NN} $ 
    can be constructed explicitly by a CBC construction in $ O(smb^m + s^2b^m) $ operations.
    \end{theorem}
    \begin{proof}
        For the case $ \calY = \RR $, this is proved in \cite[Theorem 4.1]{LoScSt:DLS22}. 
        Otherwise, let $ v\in \calY $ arbitrary such that $ \norm{v}_{\calY} = 1 $. 
        Then, 
$ \partial_{\bsy}^{\bsnu}\langle F(\bsy),v \rangle = \langle \partial_{\bsy}^{\bsnu}F(\bsy),v \rangle $ 
for all $ \bsnu\in \calF $ implies $ \norm{\langle F,v \rangle}_{s,\alpha,\bsgamma} \le  \norm{F}_{s,\alpha,\bsgamma} $. 
    Hence, we reduced to the case $ \calY = \RR $ and by linearity of the integral and the quadrature we conclude.
    \end{proof}

    \begin{corollary}\label{LoScSt:coro:paramregbip}
        Assume that the PDE \eqref{LoScSt:eq:PDE} satisfies \eqref{LoScSt:eq:Param} and \eqref{LoScSt:eq:ellp}.
        Then we can construct polynomial lattice rules $ (P_m)_{m\in\N}$ (depending on $ \bsb $)
        in $ O(smb^m + s^2b^m) $ operations
        such that for all $ \eps > 0 $ it holds for the BIP
        \begin{equation*}
            Z - Z_m = Z_{m} - Z_{m-1} + O(b^{-2m+\eps}), \qquad
            Z' - Z'_m = Z'_{m} - Z'_{m-1} + O(b^{-2m+\eps}),
        \end{equation*}
        as $m\to\infty$. The hidden constants in $O(\cdot)$ 
        are independent of $ s,m \in \NN $.
    \end{corollary}
    \begin{proof}
        From \cite[Section 4.1]{LoScSt:SS14}, \cite[Theorem 3.1]{LoScSt:DLS16} and 
    $ \bsnu! :=\prod_{j\in \operatorname*{supp}(\bsnu)} \nu_j! \le |\bsnu|!$, 
    we have that $ \sup_{s}\norm{\Theta}_{s,\alpha,\bsgamma} + \sup_{s}\norm{\Theta'}_{s,\alpha,\bsgamma} < \infty $ 
    for $ \alpha = 1 + \lfloor 1/p \rfloor $ and  some SPOD weights \eqref{LoScSt:def:spod} 
    defined by a sequence $ \bsbeta \sim \bsb $ and $ n = 0 $. 
    Hence we can apply Theorem \ref{LoScSt:thm:apostqmc} and conclude.
    \end{proof}

    A similar result can be given for the OCP,
    based on the results in \cite{LoScSt:GKKSS21,LoScSt:GKKSS22}.
    \begin{corollary}\label{LoScSt:coro:paramregoc}
        Assume that the PDE \eqref{LoScSt:eq:PDE} satisfies \eqref{LoScSt:eq:Param} and \eqref{LoScSt:eq:ellp}.
        Then we can construct polynomial lattice rules $ (P_m)_{{m\in \N}} $ (depending on $ \bsb $) 
        in $ O(smb^m + s^2b^m) $ operations
        %\todo{[complexity? ]}
        such that for all $ \eps > 0 $ it holds for the OCP
        \begin{equation*}
            Z - Z_m = Z_{m} - Z_{m-1} + O(b^{-2m+\eps}), \qquad
            Z' - Z'_m = Z'_{m} - Z'_{m-1} + O(b^{-2m+\eps}),
        \end{equation*}
        as $m\to\infty$. 
        The hidden constants in $O(\cdot)$ 
        are independent of $ s,m \in \NN $.
    \end{corollary}
\begin{proof}
Applying the result of \cite[Lemma 4.6]{LoScSt:GKKSS21} in \cite[Theorems 5.4 and 5.6]{LoScSt:GKKSS22},
we conclude that 
$ \sup_{s}\norm{\Theta}_{s,\alpha,\bsgamma} + \sup_{s}\norm{\Theta'}_{s,\alpha,\bsgamma} < \infty $ 
for $ \alpha = 1 + \lfloor 1/p \rfloor $ 
and for some SPOD weights \eqref{LoScSt:def:spod} defined by a sequence $ \bsbeta \sim \bsb $ and $ n = 2 $.
\end{proof}
%%%%%%%%%%%%%%%%%%%%%%%%%%%%%%%%%%%%%%%%%%%%%%%%%%%%%%%%%%%%%%%%%%%%%%%%%%%%%%%%%%%%%%%%%%%%%%%%%%%%%%%%%%
    \subsection{Ratio error estimator}
\label{LoScSt:sec:RatEst}
%%%%%%%%%%%%%%%%%%%%%%%%%%%%%%%%%%%%%%%%%%%%%%%%%%%%%%%%%%%%%%%%%%%%%%%%%%%%%%%%%%%%%%%%%%%%%%%%%%%%%%%%%%
    \begin{proposition}\label{LoScSt:prop:bipQMCest}
        Assume that the PDE \eqref{LoScSt:eq:PDE} satisfies \eqref{LoScSt:eq:Param} and \eqref{LoScSt:eq:ellp}.
        Then, for the BIP and the OCP, we can construct polynomial lattice rules $ (P_m)_{{m\in\N}}$ such that 
        \begin{equation}
        \frac{Z'}{Z} - \frac{Z'_{m}}{Z_{m}}  
         = 
        \frac{1}{bZ_{m} - Z_{{m-1}}} \left ( \frac{ Z_{{m-1}}Z'_{m} - Z_{m} Z'_{{m-1}}}{  Z_{m} }  \right ) 
         +  O(b^{-2m+ \eps}),
        \end{equation}
        holds for all $ \eps>0$ as $ m\to \infty $.
        The hidden constant in $O(\cdot)$ is independent of $ s,m $.
    \end{proposition}
    \begin{proof}
        First,
        $ Z - Z_{m} \to 0 $ as $ m\to \infty $ implies that $ Z_{m} > Z/2 > 0 $ for $ m $ sufficiently large.
        Next, due to either Corollary \ref{LoScSt:coro:paramregbip} or \ref{LoScSt:coro:paramregoc}, 
        for all $ \eps>0 $ we have
        \begin{align*}
            &\frac{Z'}{Z} - \frac{Z'_{m}}{Z_{m}} \\
%            &= \frac{Z' Z_{m} - Z Z'_{m}}{Z Z_{m}}
%            \\
            & = \frac{(Z' - Z'_{m}) Z_{m} - (Z - Z_{m}) Z'_{m}}{(Z - Z_{m} )  Z_{m} + Z_{m}^2}
            \\
            & =  \frac{\frac{1}{b-1}(Z'_{m} - Z'_{{m-1}} + O(b^{-2m+ \eps})) Z_{m} - \frac{1}{b-1}(Z_{m} - Z_{{m-1}} + O(b^{-2m+ \eps})) Z'_{m}}{\frac{1}{b-1}(Z_{m} - Z_{{m-1}} + O(b^{-2m+ \eps}) )  Z_{m} + Z_{m}^2}
            \\
            & = \frac{\frac{1}{b-1}(Z'_{m} - Z'_{{m-1}})}{\frac{1}{b-1}(bZ_{m} - Z_{{m-1}}) + O(b^{-2m+ \eps}) } \\ &\quad
            - \frac{\frac{1}{b-1}(Z_{m} - Z_{{m-1}} ) Z'_{m}}{(\frac{1}{b-1}(bZ_{m} - Z_{{m-1}}) + O(b^{-2m+ \eps}) )  Z_{m} } +  O(b^{-2m+ \eps}).
            %\\
            %
            %&\le \frac{1}{\abs{\abs{Z_{m}} - \abs{Z_{m} - Z_{{m-1}}} }} \left (\abs{Z'_{m} - Z'_{{m-1}}} + \frac{\abs{Z_{{m-1}}}}{\abs{Z_{m}}} \abs{Z_{m} - Z_{{m-1}}}\right ) + O(b^{-2m + \eps}),
        \end{align*}
        Clearly $ A_m := \frac{1}{b-1}(bZ_{m} - Z_{{m-1}}) \to Z $ as $ m\to \infty $, so that $ A_m > Z/2>0 $ for $ m $ sufficiently large.
        Hence, by a geometric sum argument,
        collecting in $ B_m $ all terms contained in
        $ O(b^{-2m+\eps}) $ in the denominator,
        we obtain for $ m \to\infty $ that %sufficiently large
        \begin{equation*}
            \frac{1}{A_m + B_m} 
            =  
             \frac{1}{A_m}\sum_{k=0}^{\infty}(-1)^k \left (\frac{B_m}{A_m}\right )^k 
            = 
             \frac{1}{A_m} + O(b^{-2m+\eps}).
        \end{equation*}
        Therefore, as $m\to\infty$
        \begin{equation}
            \frac{Z'}{Z} - \frac{Z'_{m}}{Z_{m}}  
            = 
            \frac{\frac{1}{b-1} (Z'_{m} - Z'_{{m-1}})}{A_m} - \frac{\frac{1}{b-1}(Z_{m} - Z_{{m-1}} ) Z'_{m}}{A_m  Z_{m} } 
            +  O(b^{-2m+ \eps}),
        \end{equation}
        which, upon rearranging the terms,  is the claim.
    \end{proof}
    Since $ \dfrac{Z'}{Z} - \dfrac{Z'_{m}}{Z_{m}} = O(b^{-m})  \gg O(b^{-2m + \eps}) $, 
    Proposition \ref{LoScSt:prop:bipQMCest} states that
    \begin{equation}\label{qmcerrest}
        E_{b^m}(\Theta',\Theta) := \frac{1}{bZ_{m} - Z_{{m-1}}} \left ( \frac{ Z_{{m-1}}Z'_{m} - Z_{m} Z'_{{m-1}}}{  Z_{m} }  \right )
    \end{equation}
    is a computable, asymptotically exact error estimator.
    %%%%%%%%%%%%%%%%%%%%%%%%%%%%%%%%%%%%%%%%%%%%%%%%%%%%%%%%%%%%%%%%%%%%%%%%%%%%%%%%%%%%%%%%%%%%%%%%%%
    \section{A-posteriori FEM error estimation}
    \label{LoScSt:sec:GalDis}
    %%%%%%%%%%%%%%%%%%%%%%%%%%%%%%%%%%%%%%%%%%%%%%%%%%%%%%%%%%%%%%%%%%%%%%%%%%%%%%%%%%%%%%%%%%%%%%%%%%
    In practice the parametric solution $ u(\bsy) \in \calX $ is not exactly available,
    and hence $ Z_{m},Z'_{m} $ are not computable.
    For any $ \bsy\in U $,  $ u(\bsy) $
    will be approximated by the
    corresponding Galerkin discretizations $ u_{h}(\bsy) \in \calX_h$.
    \subsection{Ratio error estimator}
    Let $ \Theta_{h}, \Theta'_{h}, Z_{h}, Z'_{h} $
    be defined replacing $ u $ and $ q $ by $ u_{h}, q_{h} $
    in the definitions of $ \Theta, \Theta' ,Z, Z' $, respectively.
    Similarly, let $ Z_{m,h}, Z'_{m,h} $
    be defined replacing $ u $ and $ q $ by $ u_{h}, q_{h} $
    in the definitions of $ Z_{m}, Z'_{m} $, respectively.
    %\begin{proposition}
    %    If $ \Theta_{\calT}, \Theta'_{\calT} $ are $ (b,p,\eps) $-holomorphic and bounded independently of $ \calT,\rho $ on any polytube with $ (b,\eps) $-admissible polyradius $ \rho $,
    %    then Proposition \ref{LoScSt:prop:bipQMCest} holds with $ Z,Z',Z_{m},Z'_{m} $ replaced by $ Z_{\calT},Z'_{\calT},Z_{m,\calT},Z'_{m,\calT} $ respectively, with constant in $ O(\cdot) $ independent of $ s, m $ and $ \calT $.
    %\end{proposition}

%    We now compute an a-posteriori estimator for the FE error
%    \begin{equation*}
%        \frac{Z'_{m}}{Z_{m}} - \frac{Z_{m,h}'}{Z_{m,h}}.
%    \end{equation*}

    \begin{proposition}\label{LoScSt:prop:bipFEMest}
        Assume there exist $ \zeta_{m,h}, \zeta'_{m,h} $ such that for some $ c >0 $ independent of $ h \in H $, $ \abs{Z_{m}  - Z_{m,h}} \le c \zeta_{m,h} $ and
        $ \norm{Z'_{m}  - Z_{m,h}'}_{\calY} \le c \zeta_{m,h}' $. Assume that $ \zeta_{m,h} \to 0 $ as $ h \to 0 $.
        Then, there exists $ h_0 \in H $ such that for all $ h \le h_0 $
        \begin{equation*}
            \norm{\frac{Z'_{m}}{Z_{m}} - \frac{Z_{m,h}'}{Z_{m,h}}}_{\calY} \le c\frac{Z_{m,h}\zeta_{m,h}' + \norm{Z_{m,h}'}_{\calY}\zeta_{m,h}}{Z_{m,h}^2 - c\zeta_{m,h} Z_{m,h}}.
        \end{equation*}
    \end{proposition}
    \begin{proof}
        Due to the limit $ \zeta_{m,h} \to 0 $ there exists $ h_1 \in H $ such that $ Z_{m,h} \ge Z_{m}/2 > 0 $ for all $ h \le h_1, h\in H $.
        Therefore, we can pick $ h_0 \in H $ such that $ Z_{m,h} > c \zeta_{m,h} $ for all $ h \le h_0, h \in H$.
        Thus
        \begin{align*}
            \norm{\frac{Z'_{m}}{Z_{m}} - \frac{Z_{m,h}'}{Z_{m,h}}}_{\calY} &=
            \norm{
                \frac{(Z'_{m} - Z'_{m,h}) Z_{m,h} - (Z_{m} - Z_{m,h}) Z'_{m,h}}{(Z_{m} - Z_{m,h} )  Z_{m,h} + Z_{m,h}^2} }_{\calY}
            \\
            &\le
            \frac{\norm{Z'_{m} - Z'_{m,h}}_{\calY} Z_{m,h} + \abs{Z_{m} - Z_{m,h}} \norm{Z'_{m,h}}_{\calY}}{Z_{m,h}^2 -\abs{Z_{m} - Z_{m,h}}   Z_{m,h}}
            \\
            &\le c \frac{Z_{m,h}\zeta_{m,h}' + \norm{Z_{m,h}'}_{\calY}\zeta_{m,h}}{Z_{m,h}^2 - c\zeta_{m,h} Z_{m,h}}.
        \end{align*}
    \end{proof}
Due to this result, we are left with the task of finding computable $ \zeta_{m,h}, \zeta'_{m,h} $ satisfying the conditions $ \abs{Z_{m}  - Z_{m,h}} \le c \zeta_{m,h} $ and
$ \norm{Z'_{m}  - Z_{m,h}'}_{\calY} \le c \zeta_{m,h}' $ for some $ c>0  $ independent of $ m,h $.
Thus,
in the following sections we will provide such error estimators for BIP and OCP.
%%%%%%%%%%%%%%%%%%%%%%%%%%%%%%%%%%%%%%%%%%%
\subsection{FEM error estimators for BIP}
\label{LoScSt:sec:FE4BIP}
%%%%%%%%%%%%%%%%%%%%%%%%%%%%%%%%%%%%%%%%%%%

For a finite collection of observation functionals
$ \calO = (\calO_1,\ldots,\calO_K) \in (\calX^*)^K$,
define
$ \norm{\calO}_{\calX^*} = \sqrt{\sum_{k=1}^{K} \norm{\calO_k}_{\calX^*}^2} $.
The starting point
to estimate the FEM error
will be the following well-known result, e.g.\ \cite{LoScSt:NSV09}.
Let $ \{\calT_h\}_{h\in H}$ be a family of shape-regular, simplicial meshes of $ D $
and let $ \mathbb{P}_k(\calT_h) $ be the set of piecewise polynomial functions
on $ \calT_h $ of degree at most $ k\in \NN_0 $ in each $ T\in \calT_h $.
Let $ \calE_h $ be the set of interior edges of all elements $ T\in \calT_h $.
We assume that $ \calX_{h} := \mathbb{P}_1(\calT_h) \cap \calX $,
that $ f\in L^2(D) $
and that $ a(\bsy) \in W^{1,\infty}(D) $.
Let $ h_T = |T|^{1/2} $ for $ T\in \calT_h $ and $ h_e $ the length of an edge $ e\in \calE_h $.
Then we define the a-posteriori error estimator
\begin{equation}\label{LoScSt:eq:residuals}
	\begin{split}
        \eta_{\bsy,h}^2 := &\sum_{T\in \calT_h} \Bigg[h_T^2\norm{f + \operatorname*{div}(a(\bsy) \nabla u_h(\bsy)) }_{L^2(T)}^2 
        \\ &\qquad\qquad
        + \frac{1}{2} \sum_{e \subseteq \partial T, e\in \calE_h} h_e \norm{\left\llbracket a(\bsy)\nabla u_h(\bsy) \right\rrbracket}_{L^2(e)}^2\Bigg ].
    \end{split}
\end{equation}
By \cite[Theorem 6.3]{LoScSt:NSV09} there exists
some $ c^* > 0 $, only depending on $ D $ and the shape regularity constant of $ \{\calT_h\}_{h\in H} $, such that for all $ \bsy\in U $ and $ h \in H $
    \begin{equation}\label{LoScSt:reliability}
        \norm{u(\bsy) - u_{h}(\bsy)}_{\calX} \le c^* \eta_{\bsy,h}.
    \end{equation}
For the important special case $\calO\in (L^2(D))^K$ and $G\in L^2(D)$ we may derive sharper estimates. 
To simplify the presentation, 
we assume here that the physical domain $ D\subseteq \RR^2 $ 
is a convex polygon (see also Remark~\ref{LoScSt:rem:non-convex} below), 
and introduce the $L^2(D)-$residual estimator
\begin{equation}\label{LoScSt:eq:residualsL2}
	\begin{split}
		\tilde{\eta}_{\bsy,h}^2 := &\sum_{T\in \calT_h} \Bigg[h_T^4\norm{f_{m,h}^* + \operatorname*{div}(a(\bsy) \nabla u_h(\bsy)) }_{L^2(T)}^2 
		\\ &\qquad\qquad
		+ \frac{1}{2} \sum_{e \subseteq \partial T, e\in \calE_h} h_e^3 \norm{\left\llbracket a(\bsy)\nabla u_h(\bsy) \right\rrbracket}_{L^2(e)}^2\Bigg ].
	\end{split}
\end{equation}
The additional factors $ h_T,h_e $ are derived from a standard duality argument,
see, e.g. \cite[Section 1.11]{LoScSt:Ver13BB}.
Then, there exists some $ c^* > 0 $ depending only on $ D $ 
and the shape regularity constant of $ \{\calT_h\}_{h\in H} $, such that for all $ \bsy\in U $ and $ h \in H $
\begin{equation}\label{LoScSt:reliabilityL2}
	\norm{u(\bsy) - u_{h}(\bsy)}_{L^2(D)} \le c^*    \tilde{\eta}_{\bsy,h}.
\end{equation}
    \begin{lemma}\label{LoScSt:lem:bipFEMest1}
        Fix a regular mesh $ \calT_h $ of simplices in $D$ and 
        a parameter vector $ \bsy \in U $ and assume \eqref{LoScSt:reliability}. 
        Then
        \begin{equation*}
            \abs{\Theta(\bsy) - \Theta_{h}(\bsy)} \le 
            \Theta_{h}(\bsy) (e^{\chi_{\bsy,h}} -1)=:\zeta_{\bsy,h}
         \end{equation*}
     	holds for 
     	\begin{equation*}    	
             \chi_{\bsy,h}:=\norm{\Gamma^{-1/2}\calO}_{\calX^*}\left [\abs{\delta- \calO(u_{h}(\bsy))}_{\Gamma} + \frac{1}{2} \norm{\Gamma^{-1/2}\calO}_{\calX^*} c^* \eta_{\bsy,h}\right ] c^* \eta_{\bsy,h}.
        \end{equation*}
    	Furthermore, if $\calO\in (L^2(D))^K$ and $G\in L^2(D)$, then 
   		\begin{equation*}
   			\abs{\Theta(\bsy) - \Theta_{h}(\bsy)} \le 
   			\Theta_{h}(\bsy) (e^{\tilde \chi_{\bsy,h}} -1):=\tilde \zeta_{\bsy,h}
   		\end{equation*}
    	holds for  
    	\begin{equation*}    	
    		\tilde{\chi}_{\bsy,h}:=\norm{\Gamma^{-1/2}\calO}_{L^2(D)}\left [\abs{\delta- \calO(u_{h}(\bsy))}_{\Gamma} + \frac{1}{2} \norm{\Gamma^{-1/2}\calO}_{L^2(D)} c^* \tilde{\eta}_{\bsy,h}\right ] c^* \tilde{\eta}_{\bsy,h}.
    	\end{equation*}
    \end{lemma}
    \begin{proof}
    	We define $\Delta_h(\bsy) := -\frac12\abs{\delta - \calO(u(\bsy))}_{\Gamma}^2 + \frac12\abs{\delta - \calO(u_{h}(\bsy))}_{\Gamma}^2$ to obtain 
    	\begin{align*}
    		|\Theta(\bsy) - \Theta_{h}(\bsy)| &
    		= |\Theta_{h}(\bsy)(e^{\Delta_h(\bsy)} -1)|  
    		\le \Theta_{h}(\bsy) (e^{|\Delta_h(\bsy)|} -1).
    	\end{align*} 
    	The first part of the claim now follows with~\eqref{LoScSt:reliability} since
    	\begin{align*}
    		|\Delta_h(\bsy)|&
    		\le \frac12 \norm{\Gamma^{-1/2}\calO}_{\calX^*} \abs{2\delta - \calO(u(\bsy)+u_{h}(\bsy))}_{\Gamma} \norm{u(\bsy) - u_{h}(\bsy)}_{\calX} 
                \\
    		&\le \norm{\Gamma^{-1/2}\calO}_{\calX^*}\abs{\delta - \calO(u_{h}(\bsy))}_{\Gamma} \norm{u(\bsy) - u_{h}(\bsy)}_{\calX} 
                \\
    		& \qquad + \frac12 \norm{\Gamma^{-1/2}\calO}_{\calX^*}^2\norm{u(\bsy) - u_{h}(\bsy)}_{\calX}^2.
    	\end{align*}
    	The second part of the proof follows analogously by replacing $\calX$ by $L^2(D)$ 
        and using~\eqref{LoScSt:reliabilityL2} instead of~\eqref{LoScSt:reliability}.
    \end{proof}

    \begin{lemma}\label{LoScSt:lem:bipFEMest2}
        Fix a regular mesh of simplices $ \calT_h $ and $ \bsy \in U $ and assume \eqref{LoScSt:reliability}.

        Then there exists a constant $c^*>0$ %which is independent of $\bsy\in U$ and $h\in H$, 
        such that for all $\bsy\in U$ and $h\in H$
        \begin{equation*}
            \abs{\Theta'(\bsy) - \Theta'_{h}(\bsy)} \le \norm{G}_{\calX^*} 
            \left (c^* \eta_{\bsy,h} \Theta_{h}(\bsy) e^{\chi_{\bsy,h}} 
            + \zeta_{\bsy,h}\norm{u_{h}(\bsy)}_{\calX}\right ) =: \zeta_{\bsy,h}'.
        \end{equation*}
    	Furthermore, if $\calO\in (L^2(D))^K$ and $G\in L^2(D)$, then 
    	\begin{equation*}
    		\abs{\Theta'(\bsy) - \Theta'_{h}(\bsy)} \le \norm{G}_{L^2(D)} 
    		\left (c^* \tilde{\eta}_{\bsy,h} \Theta_{h}(\bsy) e^{\tilde{\chi}_{\bsy,h}} 
    		+ \tilde{\zeta}_{\bsy,h}\norm{u_{h}(\bsy)}_{L^2(D)}\right ) =: \tilde{\zeta}_{\bsy,h}'.
    	\end{equation*}
    \end{lemma}
    \begin{proof}
        Since $ \Theta'(\bsy) = G(u(\bsy)) \Theta(\bsy) = G(u(\bsy) \Theta(\bsy) ) $,
        we get
        \begin{align*}
            \abs{\Theta'(\bsy) - \Theta'_{h}(\bsy)} &\le \norm{G}_{\calX^*} \norm{u(\bsy) \Theta(\bsy) - u_{h}(\bsy) \Theta_{h}(\bsy)}_{\calX} 
            \\
            &\le \norm{G}_{\calX^*} ( \norm{u(\bsy) - u_{h}(\bsy)}_{\calX} \Theta(\bsy) 
              + \norm{u_{h}(\bsy)}_{\calX} \abs{\Theta(\bsy) - \Theta_{h}(\bsy)} ),
        \end{align*}
        and hence the claim follows with Lemma~\ref{LoScSt:lem:bipFEMest1} since $\Theta(\bsy)\le \Theta_{h}(\bsy) e^{\chi_{\bsy,h}}$.
        The second part of the claim follows analogously by the second part of Lemma \ref{LoScSt:lem:bipFEMest1}.
    \end{proof}
\begin{remark}\label{LoScSt:rmk:GOA}
We remark that the estimates in the proof of Lemma~\ref{LoScSt:lem:bipFEMest2} are
conservative: we used that $K\ge1$ in Lemma \ref{LoScSt:lem:bipFEMest1}. 
For $K=1$, i.e. for a single observation functional,
goal-oriented AFEM results from \cite{LoScSt:GOAFEM}, and the references therein,
may be utilized to obtain sharper a-posteriori error bounds.

\end{remark}
    We define $ \zeta_{m,h} $ by averaging $ \zeta_{\bsy,h} $ for $ \bsy\in P_m $, that is $ \zeta_{m,h} := \frac{1}{b^m}\sum_{\bsy\in P_m} \zeta_{\bsy,h} $. Then
    we obtain from Lemma \ref{LoScSt:lem:bipFEMest1}
    \begin{equation}\label{LoScSt:estZ}
        \abs{Z_{m} - Z_{m,h}} \le \frac{1}{b^m}\sum_{\bsy\in P_m}\abs{\Theta(\bsy) - \Theta_{h}(\bsy)} \le  \zeta_{m,h}.
    \end{equation}
    Analogously, with $ \zeta'_{m,h} = \frac{1}{b^m}\sum_{\bsy\in P_m} \zeta'_{\bsy,h} $, Lemma \ref{LoScSt:lem:bipFEMest2} implies
    \begin{equation}\label{LoScSt:estZprime}
        \abs{Z'_{m} - Z'_{m,h}} \le \frac{1}{b^m}\sum_{\bsy\in P_m}\abs{\Theta'(\bsy) - \Theta'_{h}(\bsy)} \le  \zeta'_{m,h}.
    \end{equation}
    In particular, if we construct $ \calT_h $ such that it also holds 
    $ \eta_{\bsy,h} \to 0 $ as $ h \to 0 $ for all $ \bsy\in U $, the estimates~\eqref{LoScSt:estZ} and~\eqref{LoScSt:estZprime} verify the hypotheses of Proposition \ref{LoScSt:prop:bipFEMest} with $ c = 1 $.

%%%%%%%%%%%%%%%%%%%%%%%%%%%%%%%%%%%%%%%%%%%%%%%%%%%
\subsection{FEM error estimators for OCP with entropic risk}
\label{LoScSt:sec:FEOCEntrRsk}
%%%%%%%%%%%%%%%%%%%%%%%%%%%%%%%%%%%%%%%%%%%%%%%%%%%
    In the case of OCP
    we require error estimates for the parametric state
    at the discrete optimal control 
    $ f_{m,h}^* $, i.e.\ $ u(\bsy) = \calL^{\bsy}f_{m,h}^* $,
    and for
    the corresponding adjoint state 
    $ q(\bsy) = \alpha_1  \calL^{\bsy}(\calL^{\bsy}f_{m,h}^* - \hat{u})$.
    The error will be measured in the $ L^2(D)$-norm.
    Again we assume that $ \calX_{h} := \mathbb{P}_1(\calT_h) \cap \calX $,
    that $ f\in L^2(D) $ and $ a(\bsy) \in W^{1,\infty}(D) $, 
    and that $ D\subseteq \RR^2 $ is a convex polygon.

    \begin{lemma}\label{LoScSt:lem:ocFEMest1}
        Fix a mesh $ \calT_h $ and $ \bsy\in U $ and impose \eqref{LoScSt:reliabilityL2}. 
        With the notation of Proposition \ref{LoScSt:prop:upboundfstar} and $\tilde{\eta}_{\bsy,h}$ defined as in~\eqref{LoScSt:eq:residualsL2}, 
        we have
        \begin{align*}
            \abs{\Theta(\bsy) - \Theta_{h}(\bsy)} 
			\le \Theta_{h}(\bsy) (e^{\chi_{\bsy,h}} -1)
            =: \zeta_{\bsy,h},
        \end{align*}
    	where 
    	\begin{equation*}    	
    		\chi_{\bsy,h}:=\theta c^* \left (\frac{\alpha_1}{2} \tilde{\eta}_{\bsy,h}^2 + \alpha_1\norm{\calL_{h}^{\bsy}f_{m,h}^* - \hat{u}}_{L^2(D)}\tilde{\eta}_{\bsy,h}\right ).
    	\end{equation*}
    \end{lemma}
    \begin{proof}
        By twofold application of the triangle inequality we have
        \begin{align*}
            \abs{\Phi_{f_{m,h}^*}(\bsy) - \Phi_{h,f_{m,h}^*}(\bsy)} &\le \frac{\alpha_1}{2} \norm{(\calL^{\bsy} - \calL_{h}^{\bsy} )f_{m,h}^*}_{L^2(D)}^2
            \\ &\qquad+ \alpha_1\norm{\calL_{h}^{\bsy}f_{m,h}^* - \hat{u}}_{L^2(D)}\norm{(\calL^{\bsy} - \calL_{h}^{\bsy} )f_{m,h}^*}_{L^2(D)} \\
            &\le c^* \left (\frac{\alpha_1}{2} \tilde{\eta}_{\bsy,h}^2 + \alpha_1\norm{\calL_{h}^{\bsy}f_{m,h}^* - \hat{u}}_{L^2(D)}\tilde{\eta}_{\bsy,h}\right ) 
            =: c^* \xi_{\bsy,h}.
        \end{align*}
        Note that $ \xi_{\bsy,h} $ is computable due to Remark \ref{LoScSt:rmk:fstarmh}.
		Then, it follows with $\Delta_{h}(\bsy):=\theta(\Phi_{f_{m,h}^*}(\bsy) - \Phi_{h,f_{m,h}^*}(\bsy))$ that
        \begin{align*}
             \abs{\Theta(\bsy) - \Theta_{h}(\bsy)} 
             = \abs{\Theta_{h}(\bsy)(e^{\Delta_{h}(\bsy)}-1)} 
             \le \Theta_{h}(\bsy)(e^{|\Delta_{h}(\bsy)|}-1).
        \end{align*}
    \end{proof}
Using a residual estimator of the form \eqref{LoScSt:eq:residualsL2} for the adjoint problem yields
\begin{equation}\label{LoScSt:reliabilityL2dual}
    \norm{q(\bsy) - q_{h}(\bsy)}_{L^2(D)} \le 2  \max(c^*,1) c^*    \dbtilde{\eta}_{\bsy,h},
\end{equation}
where
\begin{align}
	\label{LoScSt:eq:residualsL2dual}
    \dbtilde{\eta}_{\bsy,h}^2 &:= 
    {\alpha_1^2}\sum_{T\in \calT_h} \Bigg [h_T^4\norm{u_h(\bsy) - \hat{u} + \operatorname*{div}(a(\bsy) \nabla q_h(\bsy)) }_{L^2(T)}^2 
    \\ &\qquad\qquad\qquad
    + \frac{1}{2} \sum_{e \subseteq \partial T, e\in \calE_h} h_e^3 \norm{\left\llbracket a(\bsy)\nabla q_h(\bsy) \right\rrbracket}_{L^2(e)}^2\Bigg ] \nonumber 
    + (\max_{T\in \calT_h } h_T^4) \tilde{\eta}_{\bsy,h}^2. 
\end{align}
    \begin{lemma}\label{LoScSt:lem:ocFEMest2}
        Let $ \calY = L^2(D) $. Fix a mesh $ \calT_h $ and $ \bsy\in U $ and impose \eqref{LoScSt:reliabilityL2} and \eqref{LoScSt:reliabilityL2dual}. With the notation of Proposition \ref{LoScSt:prop:upboundfstar}, we have
        \begin{equation*}
            \norm{\Theta'(\bsy) - \Theta'_{h}(\bsy)}_{\calY} \le \zeta_{\bsy,h} \norm{q_{h}(\bsy)}_{\calY} 
            + 2c^*\Theta_h(\bsy)e^{\chi_{\bsy,h}}\dbtilde{\eta}_{\bsy,h}
             =: \zeta_{\bsy,h}'.
        \end{equation*}
    \end{lemma}
\begin{proof}
    As in the proof of Lemma \ref{LoScSt:lem:bipFEMest2}, we obtain by $\Theta(\bsy)\le \Theta_h(\bsy)e^{\chi_{\bsy,h}}$ that
    \begin{align*}
        \norm{q(\bsy) \Theta(\bsy) - q_{h}(\bsy) \Theta_{h}(\bsy)}_{\calY}
        &\le \abs{\Theta(\bsy) - \Theta_{h}(\bsy)}\norm{q_{h}(\bsy)}_{\calY} \\
        &\quad +\Theta(\bsy) \norm{q(\bsy) - q_{h}(\bsy) }_{\calY} \\
        & \le \zeta_{\bsy,h} \norm{q_{h}(\bsy)}_{\calY} + 2c^*\Theta_h(\bsy)e^{\chi_{\bsy,h}}\dbtilde{\eta}_{\bsy,h}.
    \end{align*}
\end{proof}
\begin{remark}\label{LoScSt:rem:non-convex}
    If $ D\subseteq \RR^2 $ is a non-convex polygon,
    the reliability assumption \eqref{LoScSt:reliabilityL2}
    and the corresponding definitions \eqref{LoScSt:eq:residualsL2}, \eqref{LoScSt:eq:residualsL2dual} of $ \tilde{\eta}_{\bsy,h},\dbtilde{\eta}_{\bsy,h} $ must be adapted
    by using weighted $ L^2 $ norms, with weights near the re-entrant corners. We refer to \cite[Theorem 3.1]{LoScSt:Wi07} for a precise result in the case of the Poisson equation.
\end{remark}
%%%%%%%%%%%%%%%%%%%%%%%%%%%%%%%%%%%%%%%%%%%
    \section{Combined QMC-FEM estimator} \label{LoScSt:sec:combined}
%%%%%%%%%%%%%%%%%%%%%%%%%%%%%%%%%%%%%%%%%%%

    In view of Propositions \ref{LoScSt:prop:bipQMCest} and \ref{LoScSt:prop:bipFEMest}
    we employ the \emph{computable a-posteriori estimator}
    \begin{equation}\label{LoScSt:allest}
        EST_{b^m,h} :=\norm{E_{b^m}(\Theta'_{h},\Theta_{h})}_{\calY} + \frac{Z_{m,h}\zeta_{m,h}' + \norm{Z_{m,h}'}_{\calY}\zeta_{m,h}}{Z_{m,h}^2 - \zeta_{m,h} Z_{m,h}}.
    \end{equation}
    Note that the QMC error estimator $ \norm{E_{b^m}(\Theta',\Theta)}_{\calY} $
    derived from Proposition \ref{LoScSt:prop:bipQMCest} is itself approximated by
    the computable expression $ \norm{E_{b^m}(\Theta_h',\Theta_h)}_{\calY} $.
    In the next proposition we precise that the additional error committed
    due to this extra approximation is of higher asymptotic order, as $m\to\infty$.
    We equip the set $ C^0(U,\calY) $
    with the norm 
    $ \norm{F}_{\infty} = \sup_{\bsy\in U} \norm{F(\bsy)}_{\calY}$.
    \begin{proposition}\label{LoScSt:prop:extraerror}
        Fix a family of regular meshes $ \{\calT_h\}_{h\in H} $
        such that for some $ \tilde{C} > 0 $ independent of $ s\in \NN,h\in H $
        and some SPOD weights \eqref{LoScSt:def:spod},
        it holds
        \begin{equation}\label{LoScSt:ctilde}
            \max(\norm{\Theta}_{s,\alpha,\bsgamma}, \norm{\Theta'}_{s,\alpha,\bsgamma},\norm{\Theta_h}_{s,\alpha,\bsgamma},\norm{\Theta_h'}_{s,\alpha,\bsgamma})
        \le \tilde{C}.
        \end{equation}
        Assume that the spaces $ \{\calX_{h}\}_h $ are contained in $ \calX $
        and that they are selected so that
        $ \norm{\Theta - \Theta_{h}}_{\infty} \to 0 $ as $ h\to 0 $.
        Then we can construct a sequence of polynomial lattices 
        $ (P_m)_{m\in \NN} $ in $ O(smb^m + s^2b^m) $ operations
        such that,
        for some $ h_0 \in H $ and some constant $ C>0 $ (independent of $ m, h , s $)
        we have for any $ h < h_0 $ that
        \begin{align*}
        &\norm{E_{b^m}(\Theta',\Theta) - E_{b^m}(\Theta'_{h},\Theta_{h})}_{\calY}
        \\
        & \qquad  \le C b^{-m}(\norm{\Theta - \Theta_{h}}_{\infty} + \norm{\Theta' - \Theta'_{h}}_{\infty} + \norm{\Theta - \Theta_{h}}_{s,\alpha,\bsgamma} + \norm{\Theta' - \Theta'_{h}}_{s,\alpha,\bsgamma} ).
    \end{align*}
    
    \end{proposition}
\begin{proof}
    Throughout the proof, $ C>0 $ is a generic constant independent of $ m,h,s $. We compute the difference of the numerators
    \begin{align*}
        \Delta_1 &:= Z_{m-1}Z_{m}' - Z_{m}Z_{m-1}' - Z_{m-1,h}Z_{m,h}' + Z_{m,h}Z_{m-1,h}' \\
        & = -(Z_m - Z_{m-1})(Z'_{m} - Z'_{m,h}) - (Z_{m} - Z_{m,h} - Z_{m-1} + Z_{m-1,h})Z'_{m,h} \\
        & \qquad + (Z_{m} - Z_{m,h})(Z'_m - Z'_{m-1}) + Z_{m,h}(Z'_{m} - Z'_{m,h} - Z'_{m-1} + Z'_{m-1,h}).
        %& = (Z_{m,h} - Z_{m})Z_{m-1}' -(Z_{m-1,h}-Z_{m-1})Z_{m}' + (Z_{m-1,h}'-Z_{m-1}')Z_{m} - (Z_{m,h}' - Z_{m}')Z_{m-1} \\
        %& \qquad + (Z_{m,h} - Z_{m})(Z_{m-1,h}' - Z_{m-1}') - (Z_{m-1,h} - Z_{m-1})(Z_{m,h}' - Z_{m}').
        %& = (Z_{m,h} - Z_{m})O(b^{-m}) + (Z_{m,h} - Z_{m} - Z_{m-1,h} + Z_{m-1})Z_{m}' + (Z_{m}' - Z_{m,h}')O(b^{-m})   \\
        %&\qquad + (Z_{m}' - Z_{m,h}'-Z_{m-1}'  + Z_{m-1,h}')Z_{m} .% + (Z_{m,h} - Z_{m})O(b^{-m}) + O(b^{-m})(Z_{m,h}' - Z_{m}').
    \end{align*}
    We have $ \abs{Z_{m} - Z_{m,h}} \le \norm{\Theta - \Theta_{h}}_{\infty} $ and $ \norm{Z'_{m} - Z'_{m,h}}_{\calY} \le \norm{\Theta' - \Theta'_{h}}_{\infty} $.
    From Theorem~\ref{LoScSt:thm:apostqmc}, 
    we know that $ Z_{m} = Z_{m-1} + \norm{\Theta}_{s,\alpha,\bsgamma}O(b^{-m})$ and 
    $ Z'_{m} = Z'_{m-1} + \norm{\Theta'}_{s,\alpha,\bsgamma}O(b^{-m}) $ as $ m \to \infty $, 
    with hidden constants in $ O(\cdot) $ independent of $ s,m,h $.
    Furthermore 
    $ Z_{m} - Z_{m,h} - Z_{m-1} + Z_{m-1,h} =  \norm{\Theta - \Theta_{h}}_{s,\alpha,\bsgamma} O(b^{-m}) $, and $ Z'_{m} - Z'_{m,h} - Z'_{m-1} + Z'_{m-1,h}  = \norm{\Theta' - \Theta'_{h}}_{s,\alpha,\bsgamma} O(b^{-m})$ also follow by Theorem~\ref{LoScSt:thm:apostqmc}.
    Therefore, we have
    \begin{equation*}
        \norm{\Delta_1}_{\calY} \le C b^{-m}(\norm{\Theta - \Theta_{h}}_{\infty} + \norm{\Theta' - \Theta'_{h}}_{\infty} + \norm{\Theta - \Theta_{h}}_{s,\alpha,\bsgamma} + \norm{\Theta' - \Theta'_{h}}_{s,\alpha,\bsgamma} ).
    \end{equation*}
    Next, we define 
    $T_1 := Z_{m-1,h} Z_{m,h}' - Z_{m,h}Z_{m-1,h}' $ 
    and obtain the estimate
    $ \norm{T_1}_{\calY} \le C (\norm{\Theta_h}_{s,\alpha,\bsgamma} + \norm{\Theta'_h}_{s,\alpha,\bsgamma}) b^{-m} $.
    Moreover,
    \begin{align*}
        \Delta_2 &:= (bZ_m - Z_{m-1})Z_{m} - (bZ_{m,h} - Z_{m-1,h})Z_{m,h}
        \\
        &\;= (b(Z_{m}- Z_{m,h}) + (Z_{m-1,h} - Z_{m-1}))Z_{m} + (bZ_{m,h} - Z_{m-1,h})(Z_{m} - Z_{m,h})
    \end{align*}
    gives $ \abs{\Delta_2} \le C \norm{\Theta - \Theta_{h}}_{\infty}$.
    Next, we observe that $ T_2 = (bZ_{m,h} - Z_{m-1,h})Z_{m,h} $ is bounded from below away from $ 0 $, for $ h $ sufficiently small.
    Therefore, for $ h $ sufficiently small we apply the elementary inequality
    \begin{equation*}
        \norm{\frac{T_1 + \Delta_1}{T_2 + \Delta_2} - \frac{T_1}{T_2}}_{\calY} \le \max(\norm{\Delta_1}_{\calY}, \norm{T_1\Delta_2}_{\calY}) \frac{ 1+ \abs{T_2}}{\abs{T_2}(\abs{T_2}- \abs{\Delta_2})},
    \end{equation*}
    valid for $ T_1, T_2,\Delta_1,\Delta_2 \in \RR $ with $ \abs{\Delta_2} < \abs{T_2}  $, which is satisfied since $ \abs{\Delta_2} \to 0 $ as $ h\to 0 $.
    Combining all these observations we obtain the claim.
\end{proof}

\begin{theorem}\label{LoScSt:thm:estcombined}
    For either the BIP or the OCP, assume that $ D \subseteq \RR^2 $  
    is a convex polygon and that the PDE \eqref{LoScSt:eq:PDE} satisfies 
    \eqref{LoScSt:eq:Param}, \eqref{LoScSt:eq:ellp}, $ f\in L^2(D) $ and $ \bsb'\in \ell^p(\NN), {p\in (0,1/2]} $, 
    with $ {b'}_j =  \norm{\psi_j}_{W^{1,\infty}(D)}$. 
    Let $ \calX_{h} = \mathbb{P}_1(\calT_h) \cap \calX $ for a family of shape-regular meshes $ \calT_h $ such that $ h = \max_{T\in \calT_h}h_T $.
    Then, we can construct polynomial lattices $ (P_m)_{m\in \NN} $ such that the estimator $ EST_{b^m,h} $ in \eqref{LoScSt:allest} satisfies 
    \begin{equation*}
        \norm{\frac{Z'}{Z} - \frac{Z_{m,h}'}{Z_{m,h}}}_{\calY} \le EST_{b^m,h} + O(b^{-2m+\eps} + b^{-m}h),
    \end{equation*}
     for any $ \eps > 0 $ as $ m\to \infty,\,h\to 0 $.
     The constant in $ O(\cdot) $ is independent of $ s,m$ and $h$, 
     but depends on $ \eps $.
\end{theorem}
\begin{proof}
    Since $ b_j \le b'_j $, 
    from either Corollary \ref{LoScSt:coro:paramregbip} or \ref{LoScSt:coro:paramregoc}, 
    there exist SPOD weights $ {\bsgamma'} $ as in \eqref{LoScSt:def:spod} with $ {\bsbeta' \sim \bsb'} $, 
    such that $ \sup_{s}\norm{\Theta}_{s,\alpha,{\bsgamma'}} + \norm{\Theta'}_{s,\alpha,{\bsgamma'}} <\infty $.
    Combining \eqref{LoScSt:eq:residuals} with Lemma \ref{LoScSt:lem:bipFEMest1} 
    and $ h \to 0 $ we have $ \zeta_{m,h}\to 0 $ for the BIP case for any $ m\in \NN $.
    Similarly, combining \eqref{LoScSt:eq:residualsL2} with Lemma \ref{LoScSt:lem:ocFEMest1} yields the same observation for the OCP case.
    Therefore we can apply Propositions \ref{LoScSt:prop:bipQMCest} and \ref{LoScSt:prop:bipFEMest} to get that we can construct polynomial lattices so that as $ m\to \infty $
    \begin{equation*}
        \norm{\frac{Z'}{Z} - \frac{Z_{m,h}'}{Z_{m,h}}}_{\calY} \le \norm{E_{b^m}(\Theta',\Theta)}_{\calY} + \frac{Z_{m,h}\zeta_{m,h}' + \norm{Z_{m,h}'}_{\calY}\zeta_{m,h}}{Z_{m,h}^2 - \zeta_{m,h} Z_{m,h}} + O(b^{-2m+\eps}).
    \end{equation*}
    Next, we say that $ \bsrho \in (1,\infty)^{\NN}  $ is $ (\bsb',\eps)\text{-admissible} $ if $ \sum_{j \ge 1} (\rho_j - 1)b'_j \le \eps  $, 
    see \cite{LoScSt:DLS16}.
    Then we define $ \frakT_{\bsb',\eps} = \bigcup_{\bsrho, (\bsb',\eps)\text{-adm.} } 
                      \{\bsy \in \CC^s : \operatorname{dist}(y_j, [-\tfrac12,\tfrac12]) \le \rho_j-1 \} $.
    Following the computations in \cite[Theorem 4.1]{LoScSt:DGLS17}, 
   we obtain for $ h_0\in H $ and $ \eps>0 $ sufficiently small that $ \sup_{\bsy\in \frakT_{\bsb',\eps}}\abs{\Theta(\bsy)- \Theta_h(\bsy)} + \norm{\Theta'(\bsy)- \Theta_h'(\bsy)}_{\calY} \le Ch $
   holds for all $ h \le h_0, h\in H $.
    By \cite[Theorem 3.1]{LoScSt:DLS16}, this implies that for a constant $ C $ independent of $ h,s $,
    \begin{equation*}
        \norm{\Theta - \Theta_{h}}_{\infty} 
      + \norm{\Theta' - \Theta'_{h}}_{\infty} 
      + \norm{\Theta - \Theta_{h}}_{s,\alpha,{\bsgamma'}} 
      + \norm{\Theta' - \Theta'_{h}}_{s,\alpha,\bsgamma'} \le Ch.
    \end{equation*}
Thus, \eqref{LoScSt:ctilde} and 
$ \norm{\Theta - \Theta_{h}}_{\infty}\to 0 $ 
hold, and we apply Proposition \ref{LoScSt:prop:extraerror} to conclude.
\end{proof}

\section{{Numerical experiment}} \label{LoScSt:sec:numerics}

We consider the PDE~\eqref{LoScSt:eq:PDE} on the physical domain $D:=(0,1)^2$, 
with $f\equiv 10$, and parametric diffusion coefficient given by
\begin{equation*}	
	a(x,\bsy) = \frac{1}{2}+ \sum_{j = 1}^{16} \frac{y_j}{(k_{j,1}^2+k_{j,2}^2)^2}
	\sin(k_{j,1}x_1)\sin(k_{j,2}x_2).
\end{equation*}
The pairs $(k_{j,1}, k_{j,2})\in\N^2$ are defined by the ordering of $\N^2$ such that for $j\in\N$, $k_{j,1}^2+k_{j,2}^2\le k_{j+1,1}^2+k_{j+1,2}^2$, and the ordering is arbitrary when equality holds. 
We investigate a BIP with observation functional 
$\calO = (\calO_1, \dots, \calO_4)\in (L^2(D))^4$, given by 
$\calO_k(v) := \frac{1}{0.01}\int_{I_k} v dx$
for $v\in L^2(D)$ and $k=1,\dots, 4$, where $I_1:=[0.1, 0.2]^2, I_2:=[0.1, 0.2]\times [0.8,0.9], I_3:=[0.8, 0.9]\times [0.1,0.2], I_4:=[0.8,0.9]^2$. 
We draw a (random) sample of $a$ to compute the "ground truth" of observations $\calO(\calS(a))$ on a 
sequence of regular FE meshes of triangles obtained by uniform refinement,
with $525.313$ degrees of freedom (dofs). 
We add random noise 
$ \eta \sim \calN(0,\sigma^2\calI_4) $ to the observations, 
where $\sigma$ is set as $10\%$ of the arithmetic average of $\calO(\calS(a))\in \R^4$.
The realized synthetic data is given by $\delta = (0.5205, 0.5037, 0.5443, 0.4609)^\top$. 
A well-known issue is that the performance of ratio estimators deteriorates as $\sigma\to0$, 
    i.e. as $Z\to0$, and we refer to Section~\ref{LoScSt:sec:conclusion} below for a further discussion on this matter.

Our aim is to approximate $\EE^{\pi^{\delta}}[G(u)]$ by the ratio estimator $\frac{Z'_{m,h}}{Z_{m,h}}$, 
where $G\in L^2(D)$ is given by $G(v) := \frac{1}{0.5}\int_{[0.25, 0.75]^2} v dx$ for $v\in L^2(D)$.
The FE mesh and the polynomial lattice rule, that eventually determine $h$ and $m$, 
are refined successively based on the combined estimator in~\eqref{LoScSt:allest}.
For tolerances $\tau_{\text FEM}, \tau_{\text QMC}>0$, 
we start from an initial FE mesh of $D$, that is uniformly refined until the stopping criterion
$\frac{Z_{m,h}\zeta_{m,h}' + \norm{Z_{m,h}'}_{\calY}\zeta_{m,h}}{Z_{m,h}^2 - \zeta_{m,h} Z_{m,h}}\le \tau_{\text FEM}$
is met. 
Thereafter, 
we increase the number $b^m$ of lattice points until there holds $|E_{b^m}(\Theta'_{h},\Theta_{h})|\le \tau_{\text QMC}$.
We initialize by a FE mesh with $41$ dofs and $b^{m_0}$ QMC points with base $b=2$ and $m_0=2$, 
and set the tolerances to $\tau_{\text FEM}=\tau_{\text QMC}=2^{-6}$.
To assess the total realized error, we compute a reference solution $\frac{Z'_{\text{ref}}}{Z_{\text{ref}}}$ 
by a multilevel Monte Carlo ratio estimator, 
see \cite{LoScSt:SST17}, and report absolute error $\left|\frac{Z'_{m,h}}{Z_{m,h}}-\frac{Z'_{\text{ref}}}{Z_{\text{ref}}}\right|$. 
The reference estimator uses $6$ refinement levels with $545$/$525.313$ dofs 
on the coarsest/finest level, respectively, and uniform (pseudo-) random numbers $y$. 
The number of samples is adjusted to balance statistical error and discretization bias on each level.
The experiment has been implemented in \textsc{MATLAB} using the \textsc{MooAFEM} library \cite{LoScSt:IP22} 
for the FE discretization.
All arising linear systems are solved directly by the $\backslash$-operator in \textsc{MATLAB}.

The estimated and realized errors vs. the number of iterations (in the sense of refinement steps) 
are depicted Figure~\ref{LoScSt:fig:experiment}. 
Here, the FE a-posteriori estimator gives negative values on rather coarse meshes, 
where $c^*\tilde{\eta}_{\bsy,h}>1$. 
Therefore, we discarded these "pre-asymptotic" values in the plot. 
We see that the FE a-posteriori estimator from Proposition~\ref{LoScSt:prop:bipFEMest} 
is rather conservative at first, 
but eventually approaches the actual error for finer meshes. 
The QMC estimator $\abs{E_{b^m}(\Theta_h', \Theta_h')}$ is of the same magnitude as $\sigma$ at first, 
and only two more refinement steps are needed once the FE mesh is sufficiently fine. 
The combined error estimate $EST_{b^m,h}$ aligns well with the realized error, 
as expected from our theoretical analysis. 
\begin{figure}[htb]
	\centering
	\includegraphics[scale = 0.45]{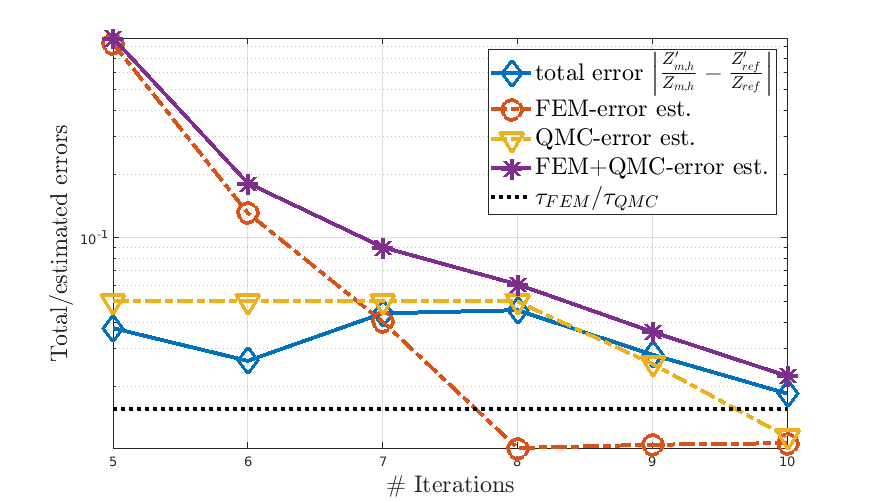}
	\caption{\small Results for the QMC-FEM ratio estimator with a-posterior ratio refinement. 
        First the FE mesh is refined until the tolerance $\tau_{\text FEM}$ is achieved (dashed w. circles),
        then the QMC a-posterior refinement takes place (dashed w. triangles). 
        The estimated error (solid w. stars) is conservative for coarse meshes,
        but eventually approaches the realized error (solid w. diamonds).}
	\label{LoScSt:fig:experiment}
\end{figure}

\section{Conclusion} \label{LoScSt:sec:conclusion}

    In this paper we outlined the construction of an a-posteriori QMC-FEM estimator,
    that allows to quantify the approximation error to a) posterior expectation in Bayesian inverse problems and b) the optimal control under the entropic risk measure.
    The estimator is computable and viable for large number of parameters $ s $ and it is asymptotically an upper bound for the errors in a) and b).
    Furthermore, the particular ratio structure $ \frac{Z'}{Z} $ of the sought quantities allows to tackle both
    the BIP and OCP, in a unified manner.
    In either case,
    we work under the assumption that the underlying model is a
    parametric elliptic PDE with affine-parametric diffusion.
    Nevertheless, the present QMC methodology to high-dimensional integration is applicable to 
    non-affine parametric PDEs with \emph{quantified, holomorphic-parametric dependence}, 
    see \cite{LoScSt:DLS16} and the references there.
    Since the error estimators 
    $ \eta_{\bsy,h},\tilde{\eta}_{\bsy,h},\dbtilde{\eta}_{\bsy,h} $
    are expressed as sums of local error contributions for $ T\in \calT_h $,
    a possible direction of research is to employ
    the presently proposed estimators $ \zeta_{\bsy,h}, \zeta'_{\bsy,h} $
    to steer an adaptive QMC-FEM algorithm \cite{LoScSt:Lo22}.
	
	The performance of ratio estimators in BIPs deteriorates in the 
        \emph{small noise limit} $\Gamma\downarrow 0$ resp. in \emph{large data regimes} with $K\gg 1$:
        the denominator $Z\to0$ in either case. 
	Nonrobustness of the ration estimators w.r. to $\Gamma$ 
        can not be overcome alone by the presented techniques.
	It may be alleviated by applying the present methodology following 
        suitable re-parametrization of the posterior, as explained e.g. in \cite{LoScSt:SS16}.
	{In view of \cite{LoScSt:DGLS17}, a-posteriori error bounds for multilevel ratio estimators are conceivable in the uniform elliptic setting of Subsection~\ref{LoScSt:sec:FwdPbm}. These extensions have not been included here for reasons of length, but will be subject to future research.
	
    %\bibliographystyle{abbrv}
    %\bibliography{references}

\begin{thebibliography}{10}
    	\small

        \bibitem{LoScSt:GOAFEM}
        R.~Becker, M.~Brunner, M.~Innerberger, J.~M. Melenk, and D.~Praetorius.
        \newblock Rate-optimal goal-oriented adaptive {FEM} for semilinear elliptic
        {PDE}s.
        \newblock {\em Comput. Math. Appl.}, 118:18--35, 2022.

        \bibitem{LoScSt:DGLS17}
        J.~Dick, R.~N. Gantner, Q.~T.~L. Gia, and Ch.~Schwab.
        \newblock Multilevel higher-order quasi-{M}onte {C}arlo {B}ayesian estimation.
        \newblock {\em Math. Mod. Meth. Appl. Sci.}, 27(5):953--995, 2017.

        \bibitem{LoScSt:DGY19}
        J.~Dick, T.~Goda, and T.~Yoshiki.
        \newblock Richardson extrapolation of polynomial lattice rules.
        \newblock {\em SIAM J. Numer. Anal.}, 57(1):44--69, 2019.

        \bibitem{LoScSt:DLS16}
        J.~Dick, Q.~T. {Le Gia}, and Ch.~Schwab.
        \newblock Higher order quasi-{M}onte {C}arlo integration for holomorphic,
        parametric operator equations.
        \newblock {\em SIAM/ASA J. Uncertain. Quantif.}, 4(1):48--79, 2016.

        \bibitem{LoScSt:DLS22}
        J.~Dick, M.~Longo, and Ch.~Schwab.
        \newblock Extrapolated polynomial lattice rule integration in computational
        uncertainty quantification.
        \newblock {\em SIAM/ASA J. Uncertain. Quantif.}, 10(2):651--686,
        2022.

        \bibitem{LoScSt:EG21BB}
        A.~Ern and J.-L. Guermond.
        \newblock {\em Finite elements {I}---{A}pproximation and interpolation},
        volume~72 of {\em Texts in Applied Mathematics}.
        \newblock Springer, Cham, 2021.

        \bibitem{LoScSt:FK11}
        H.~F\"{o}llmer and T.~Knispel.
        \newblock Entropic risk measures: coherence vs. convexity, model ambiguity, and
        robust large deviations.
        \newblock {\em Stoch. Dyn.}, 11(2-3):333--351, 2011.

        \bibitem{LoScSt:GKKSS21}
        P.~A. Guth, V.~Kaarnioja, F.~Y. Kuo, C.~Schillings, and I.~H. Sloan.
        \newblock A quasi-{M}onte {C}arlo method for optimal control under uncertainty.
        \newblock {\em SIAM/ASA J. Uncertain. Quantif.}, 9(2):354--383, 2021.

        \bibitem{LoScSt:GKKSS22}
        P.~A. Guth, V.~Kaarnioja, F.~Y. Kuo, C.~Schillings, and I.~H. Sloan.
        \newblock Parabolic {PDE}-constrained optimal control under uncertainty with
        entropic risk measure using quasi-{M}onte {C}arlo integration.
        \newblock arXiv:2208.02767, 2022.
        
        \bibitem{LoScSt:HKS21}
   L.~Herrmann, M.~Keller Ch.~Schwab.
     \newblock {Quasi-{M}onte {C}arlo {B}ayesian estimation under {B}esov
                 priors in elliptic inverse problems}.
     \newblock {Math. Comp.}, {\bf 90}(2021) 1831--1860.

        \bibitem{LoScSt:IP22}
        M. Innerberger and D. Praetorius. "MooAFEM: An object oriented Matlab code for higher-order (nonlinear) adaptive FEM." , arXiv:2203.01845, 2022.

        \bibitem{LoScSt:KS18}
        D.~P. Kouri and T.~M. Surowiec.
        \newblock Existence and optimality conditions for risk-averse {PDE}-constrained
        optimization.
        \newblock {\em SIAM/ASA J. Uncertain. Quantif.}, 6(2):787--815, 2018.

        \bibitem{LoScSt:Lo22}
        M.~Longo.
        \newblock Adaptive {Q}uasi-{M}onte {C}arlo {F}inite {E}lement {M}ethods for
        {P}arametric {E}lliptic {PDE}s.
        \newblock {\em J. Sci. Comput.}, 92(1), 2022.

        \bibitem{LoScSt:ML22DD}
        M.~Longo.
        \newblock {\em Extrapolated polynomial lattices and adaptivity in computational
            {U}ncertainty {Q}uantification}.
        \newblock PhD thesis, ETH Z\"urich, 2022.

        \bibitem{LoScSt:N92}
        H.~Niederreiter.
        \newblock Low-discrepancy point sets obtained by digital constructions over
        finite fields.
        \newblock {\em Czechoslovak Math. J.}, 42(117)(1):143--166, 1992.

        \bibitem{LoScSt:NSV09}
        R.~H. Nochetto, K.~G. Siebert, and A.~Veeser.
        \newblock Theory of adaptive finite element methods: an introduction.
        \newblock In {\em Multiscale, nonlinear and adaptive approximation}, pages
        409--542. Springer, Berlin, 2009.
		
		\bibitem{LoScSt:SST17}
		R. Scheichl, A.M. Stuart, and A.L. Teckentrup.
		\newblock Quasi-Monte Carlo and multilevel Monte Carlo methods for computing posterior expectations in elliptic inverse problems.
		\newblock {\em SIAM/ASA J. Uncertain. Quantif.}, 5(1): 493-518, 2017.
		
        \bibitem{LoScSt:SS14}
        C.~Schillings and Ch.~Schwab.
        \newblock Sparsity in {B}ayesian inversion of parametric operator equations.
        \newblock {\em Inverse Problems}, 30(6):065007, 30, 2014.
        
        \bibitem{LoScSt:SS16}
        C.~Schillings and Ch.~Schwab.
        \newblock Scaling limits in computational Bayesian inversion.
        \newblock {\em ESAIM: Mathematical Modelling and Numerical Analysis}, 50(6): 1825-1856, 2016.

        \bibitem{LoScSt:Stu10}
        A.~M. Stuart.
        \newblock Inverse problems: a {B}ayesian perspective.
        \newblock {\em Acta Numer.}, 19:451--559, 2010.

        \bibitem{LoScSt:Ver13BB}
        R.~Verf{\"u}rth.
        \newblock {\em a-posteriori error estimation techniques for finite element
            methods}.
        \newblock Numerical Mathematics and Scientific Computation. Oxford University
        Press, Oxford, 2013.

        \bibitem{LoScSt:Wi07}
        T.~P. Wihler.
        \newblock Weighted {$L^2$}-norm a-posteriori error estimation of {FEM} in
        polygons.
        \newblock {\em Int. J. Numer. Anal. Model.}, 4(1):100--115, 2007.

    \end{thebibliography}
{
    
}
\end{document}